\documentclass[a4paper,12pt]{amsart}
\usepackage{amsmath, amssymb, amsthm, xcolor,enumerate, enumitem}
\usepackage{tikz-cd}
\usepackage{bbm}
\usepackage{graphicx}
\usepackage{marginnote}
\usepackage[hidelinks]{hyperref}

\usepackage{perpage}
\newcounter{mparcnt}
\MakePerPage{mparcnt}

\newcounter{counter}

\newtheorem{theorem}[counter]{Theorem}
\newtheorem{lemma}[counter]{Lemma}

\newtheorem{definition/proposition}[counter]{Definition/Proposition}
\newtheorem{corollary}[counter]{Corollary}
\newtheorem{proposition}[counter]{Proposition}
\newtheorem*{theorem*}{Theorem}
\numberwithin{equation}{section}

\theoremstyle{definition}
\newtheorem{definition}[counter]{Definition}

\newtheorem{remark}[counter]{Remark}
\newtheorem{example}[counter]{Example}

\numberwithin{equation}{section}
\newcommand{\D}{\mathcal{D}}

\renewcommand{\O}{\mathcal{O}}

\newcommand{\G}{\mathcal{G}}
\newcommand{\Z}{\mathcal{Z}}
\newcommand{\id}{\mathrm{id}}

\newcommand{\Corr}{\mathrm{Corr}}

\newcommand{\N}{\mathbb{N}}
\newcommand{\C}{\mathcal{C}}
\newcommand{\Aut}{\mathrm{Aut}}

\newcommand{\Hilb}{\mathrm{Hilb}}
\newcommand{\Irr}{\mathrm{Irr}}

\newcommand{\End}{\mathrm{End}}
\newcommand{\Hom}{\mathrm{Hom}}

\newcommand{\dist}{\mathrm{dist}}

\title[Tensor category equivariant $\D$-stability]{Equivariant $\D$-stability for Actions of Tensor Categories}

\author{Samuel Evington}
\address{\hskip-\parindent Samuel Evington, Mathematical Institute, University of Münster, Einsteinstrasse 62, 48149 Münster, Germany}
\email{evington@uni-muenster.de}

\author{Sergio Girón Pacheco}
\address{\hskip-\parindent Sergio Girón Pacheco, Department of mathematics, KU Leuven, Celestijnenlaan 200B, 3001, Leuven, Belgium.}
\email{sergio.gironpacheco@kuleuven.be}

\author{Corey Jones}
\address{\hskip-\parindent Corey Jones, Department of Mathematics, North Carolina State University, North Carolina, USA}
\email{cmjones6@ncsu.edu}

\thanks{SE was partially supported by the Deutsche Forschungsgemeinschaft (DFG, German Research Foundation) – Project-ID 427320536 – SFB 1442, as well as under Germany's Excellence Strategy EXC 2044 390685587, Mathematics M{\"u}nster: Dynamics–Geometry–Structure, and was also partially supported by ERC Advanced Grant 834267 - AMAREC; SGP was partially supported by the Ioan and Rosemary James Scholarship awarded by St John's College and the Mathematical Institute, University of Oxford, as well as by projects G085020N and 1249225N funded by the Research Foundation Flanders (FWO); CJ was partially supported by NSF Grants DMS-2100531 and DMS-2247202}

\begin{document}
\begin{abstract}
We introduce a notion of equivariant $\D$-stability for actions of unitary tensor categories on C$^*$-algebras. We show that, when $\D$ is strongly self-absorbing, equivariant $\D$-stability of an action is equivalent to a unital embedding of $\D$ into a certain subalgebra of Kirchberg's central sequence algebra.\ We use this to show $\Z$-stability for a large class of AF-actions.
\end{abstract}
\maketitle
\numberwithin{counter}{section}
\section{introduction}
Tensorial absorption of certain C$^*$-algebras plays a crucial role in the classification of simple amenable C$^*$-algebras and the investigation of their structure.\ 
In the infinite setting, tensorial absorption of the Cuntz algebra $\O_\infty$, a property known as $\O_\infty$-stability, powers the Kirchberg--Phillips theorem (\cite{KIRCH95,Phil00,KIPI00}). 
In the finite setting, tensorial absorption of the Jiang--Su algebra $\Z$ has proven to be of fundamental importance (\cite{Jiang-Su,TO08,GLN15,EGLN15,TWW17,DIMNUC21,Class1}). 

Abstracting the key properties of the UHF algebras $\mathbb{M}_{n^\infty}$, the Cuntz algebras $\O_2$ and $\O_\infty$, and the Jiang--Su algebra $\Z$, Toms and Winter introduced the notion of a strongly self-absorbing C$^*$-algebra $\D$ and developed a unified theory of $\D$-stable C$^*$-algebras including machinery for detecting $\D$-stability in examples (\cite{TOWI07}).  
The notion of $\D$-stability has subsequently been extended to the setting of group actions on C$^*$-algebras and has been successfully used both for the classification of group actions on C$^*$-algebras (\cite{IZU04,SZ18Kirchberg,GASZ22,IZMA21}) and to identify $\Z$-stability of crossed product C$^*$-algebras (\cite{WOU23,GAHIVA22,MASA12,MASA14}).   

In this paper, we introduce and study equivariant $\D$-stability for an action of a unitary tensor category $\C$. 
Actions of tensor categories on operator algebras generalize group actions and give a characterisation of ``non-invertible'' symmetry. They have largely been considered in the setting of von Neumann algebras, where they play a fundamental role in the theory of finite index subfactors (\cite{MR1278111,MR1334479,MR1642584,MR1966524,MR3166042}). Recently there has been significant interest in extending these ideas to the setting of C*-algebras (\cite{IZU02,MR4419534,EVGI21,CHROJO22,WA90,KW00}). 

Employing the intertwining techniques developed in \cite{GINE23}, we shall prove that equivariant $\D$-stability of an action $\C\curvearrowright A$ on a C$^*$-algebra $A$ is equivalent to the existence of a unital embedding of $\D$ into an appropriate subalgebra of the Kirchberg central sequence algebra of $A$ (Theorem \ref{thm:eqMcDuff}). 
This result generalises \cite[Theorem 2.2]{TOWI07} and its group equivariant counterpart \cite[Corollary 3.8]{SZI} under the assumption that the acting group is discrete.\
We apply this characterisation to large class of AF-actions of unitary tensor categories (in the sense of \cite{CHROJO22}) and provide an easily verifiable criterion for determining equivariant $\D$-stability (Theorem \ref{thm:BYintrocentralsequence}). In particular, we show that these AF-actions are always equivariantly $\Z$-stable (Corollary \ref{cor:stationaryareZstable}).

To study equivariant $\D$-stability for group actions, the standard approach is to consider the induced group action on the algebra of central sequences, and then consider the fixed point subalgebra of the induced action (see for example \cite{MASA12,MASA14}). 
This doesn't directly generalise to actions of unitary tensor categories, as an action $\C\curvearrowright A$ may not preserve central sequences. 
For example, if $A$ is simple purely infinite, then the action may be assumed to be given by endomorphisms (in the sense of \cite[Section~4]{IZU02}). 
If the simple objects in $\C$ are not invertible, then the endomorphisms inducing the action on $A$ need not be surjective and so need not preserve central sequences.

Our solution to this problem is to directly introduce a C$^*$-subalgebra of the central sequence algebra of $A$ that plays the role of the algebra of central sequences fixed under the action of the tensor category. 
The appropriate C$^*$-subalgebra is obtained by considering bounded sequences in $A$ whose the left and right actions asymptotically agree for all the Hilbert $A$-bimodules associated to the action (see Proposition \ref{prop:fixedcentralseqalg}). 
Since the tensor unit of $\C$ is mapped to the trivial $A$-bimodule, all such sequences are necessarily central sequences. 

\subsection*{Acknowledgements} The authors would like to thank Robert Neagu and Stuart White for interesting discussions related to the topic of this paper.

\section{Preliminaries}

We assume the reader is familiar with the language of tensor categories and C$^*$-tensor categories (see for example \cite{MR3242743} and \cite[Section 1.3 and Section 2]{GINE23}). We also assume some familiarity with both the theory of Hilbert modules (see for example \cite{Hilbertmodules}) and with strongly self-absorbing C$^*$-algebra as introduced in \cite{TOWI07}.

Let $A$ be a C$^*$-algebra. 
An \emph{$A$-$A$-correspondence} is a $^*$-homomorphism $\alpha:A\rightarrow \mathcal{L}(X)$ into the adjointable operators of a right Hilbert $A$-module $X$.
This endows $X$ with the structure of a Hilbert $A$-bimodule in the sense of \cite{KWP04}.

For $x\in X$ and $a\in A$ we denote the left action of $a$ on $x$ by $a\rhd x$ and the right action of $a$ on $x$ by $x\lhd a$. We denote the internal tensor product of Hilbert bimodules by $\boxtimes$. 
We say $\alpha$ is \emph{non-degenerate} if $X = \overline{A\rhd X} = \overline{X \lhd A}$, which for $A$ unital occurs precisely when the left and right actions are unital.

Let $\Corr_0(A)$ be the C$^*$-tensor category of non-degenerate Hilbert $A$-$A$-correspondences under the flipped internal tensor product of the bimodules, i.e.\ the monoidal product of $A$-$A$-correspondences $\alpha:A\rightarrow \mathcal{L}(X)$ and $\beta:A\rightarrow \mathcal{L}(Y)$ is the $A$-$A$-correspondence given by the internal tensor product $Y \boxtimes X$ (see \cite[Section 1.2]{GINE23}). 

\begin{definition}
   An action of a C$^*$-tensor category $\mathcal{C}$ on $A$ is a C$^*$-tensor functor $(G,J):\C\rightarrow \Corr_0(A)$. 
\end{definition}

An important example of a C$^*$-tensor category is the category of finite dimensional $\Gamma$-graded Hilbert spaces, where $\Gamma$ is a discrete group. We denote this category by $\Hilb(\Gamma)$.  We write $\mathbb{C}_{\gamma}$ for $\Gamma$ graded Hilbert space that is $\mathbb{C}$ in the $\gamma$-th grading and zero elsewhere.

We say an action $(G,J):\Hilb(\Gamma)\rightarrow \Corr_0(A)$ acts by automorphisms if $G(\mathbb{C}_\gamma)\cong {}_{\alpha_\gamma}A$ for all $\gamma\in \Gamma$ where ${}_{\alpha_\gamma}A$ denotes the trivial right Hilbert $A$-module endowed with the left action $a \rhd x = \alpha_\gamma(a)x$ for some $\alpha_\gamma \in \Aut(\Gamma)$. Note $\Hilb(\Gamma)$ actions by automorphism on $A$ coincide with cocycle actions of $\Gamma$ on $A$ in the sense of \cite[Section 2.1]{IZU04}.

We now recall the theory of dual objects in C$^*$-tensor categories as introduced in \cite{LONRO97}. A C$^*$-tensor category $\mathcal{C}$ is said to be \emph{rigid} if every $X\in \C$ admits a dual object. That is, there exists $\overline{X}\in \mathcal{C}$ and maps $R:1_{\C}\rightarrow \overline{X}\otimes X$ and $\overline{R}:1_{\C}\rightarrow X\otimes \overline{X}$ such that the zig-zag relations hold. Surpressing the associator and unitor maps these relations are
\begin{align}
    (\overline{R}^*\otimes \id_X)(\id_{X}\otimes R)&=\id_X,\\
    (R^*\otimes \id_{\overline{X}})(\id_{\overline{X}}\otimes \overline{R})&=\id_{\overline{X}}. 
\end{align}
Any rigid C$^*$-tensor category that admits subobjects is semisimple and hence every object can be decomposed as a finite direct sum of simple objects. When $\C$ is semisimple, we denote by $\Irr(\mathcal{C})$ a choice of isomorphism classes for simple objects.
\begin{definition}
    A \emph{unitary tensor category} $\C$ is a rigid, semisimple, C$^*$-tensor category with simple unit object, i.e. $\End(1_{\C})\cong \mathbb{C}$. Moreover, $\C$ is called a \emph{unitary fusion category} if it has finitely many isomorphism classes of simple objects.
\end{definition}
In \cite{KWP04} Kajiwara, Pinzari and Watatani characterise which objects in $\Corr_0(A)$ have duals. It is shown in \cite[Theorem 4.13]{KWP04} that if $E$ is a non-degenerate Hilbert $A$-bimodule admitting a dual object, then $E$ admits a compatible left $A$-valued inner product making $E$ a bi-Hilbertian C$^*$-bimodule in the sense of \cite[Definition 2.3]{KWP04}. In this case, its dual is isomorphic to the contragredient bimodule $\overline{E}$.
\begin{definition}\label{def:cocyclemorph}
If $E$ is a bi-Hilbert $A$-bimodule with left inner product ${}_A\langle\cdot,\cdot\rangle$ and right inner product $\langle \cdot, \cdot\rangle_A$ then its \emph{contragredient bimodule} $\overline{E}$ is the bi-Hilbertian $A$-bimodule with underlying vector space $\overline{E}=\{\overline{\eta} : \eta\in E\}$ with the conjugate vector space structure, and left and right $A$-actions defined by
\begin{equation}
    a\rhd\overline{\eta}:=\overline{\eta\lhd a^*},\quad \overline{\eta}\lhd a:=\overline{a^*\rhd \eta}
\end{equation}
for $a\in A$ and $\eta\in E$. The left and right inner products of $\overline{E}$ are defined by
\begin{equation}
    \langle \overline{\eta},\overline{\xi}\rangle_A:={}_A \langle \eta,\xi\rangle,\quad {}_A\langle \overline{\eta},\overline{\xi}\rangle:=\langle \eta,\xi\rangle_A.
\end{equation}
for $\eta,\xi\in E$.
\end{definition}
We will call Hilbert bimodules which admit dual objects \emph{dualisable}. In this case, we may always assume that the dual is given by the contragredient bimodule.

Notions of equivariant morphisms between two actions $(G,J)$ and $(H,I)$ of a C$^*$-tensor category $\C$ on a C$^*$-algebras $A$ and $B$ have appeared in the literature under (slightly) different names (see \cite{CHROJO22,YUKAKU23,GINE23} for example). 
We shall use a definition of cocycle morphisms given by \cite[Lemma 3.8]{CHROJO22} in the unital setting and by \cite[Lemma 3.11]{GINE23} in general.

\begin{definition}
A \emph{cocycle morphism} $(\varphi,\{h^X\}_{X\in \C}):(G,J)\rightarrow (H,I)$ consists of a $^*$-homomorphism $\varphi:A\rightarrow B$ and linear maps $h^X:G(X)\rightarrow H(X)$ for each $X\in \C$ such that
\begin{enumerate}[label=\textit{(\roman*)}]
\item $h^{X}(a\rhd x \lhd a')=\varphi(a)\rhd h^{X}(x)\lhd \varphi(a')$ for any $a,a'\in A$;\label{item:linearmap1}
\item
for any morphism $f\in \Hom(X,Y)$, $H(f)\circ h^{X}=h^{Y}\circ G(f)$;\label{item:linearmap2}
\item
$\varphi(\langle x , y\rangle_A)=\langle h^{X}(x) , h^{X}(y)\rangle_B$ for any $x,y\in G(X)$; \label{item:linearmap3}
\item the diagram:\label{item:linearmap4}
\[
\begin{tikzcd}
G(Y)\boxtimes G(X)
\arrow[swap]{d}{h^{Y}\boxtimes h^{X}}
\arrow{r}{J_{X,Y}}
& G(X\otimes Y)\arrow{d}{h^{X\otimes Y}}
 \\
H(Y)\boxtimes H(X)
\arrow{r}{I_{X,Y}}
& H(X\otimes Y)
\end{tikzcd}
\]
commutes;
\item $h^{1_{\C}}:A\to B$ is given by $h^{1_{\C}}(a)=\varphi(a)$ for any $a\in A$.\label{item:linearmap5}
\end{enumerate}  
If $h^X$ are bijective for every $X\in \C$ then $(\varphi,\{h^X\}_{X\in \C})$ is called a \emph{cocycle conjugacy}.
\end{definition}

Let $A$ be a separable C$^*$-algebra and $\C$ be a semisimple C$^*$-tensor category with countably many isomorphism classes of simple objects. Cocycle morphisms $(\varphi,\{h^X\}_{X\in \C})$ and $(\psi,\{l^X\}_{X\in \C})$ from $(G,J)$ to $(H,I)$ are said to be \emph{approximately unitarily equivalent}, if there exists a sequence of unitaries in the multiplier algebra $(u_n)_{n\in \N}\in M(B)$ such that 
\begin{equation}
    \lim_{n \to \infty}\|u_n\rhd h^X(x)\lhd u_n^*-l^X(x)\| = 0
\end{equation} for all $X\in \C$ and $x\in F(X)$ (see \cite[Definition C]{GINE23}).

\section{A subalgebra of Kirchberg's central sequences}\label{sec:centralseq}
Let $A$ be a C$^*$-algebra. We set $A_\infty = \ell_\infty(A)/c_0(A)$. 
Given a Hilbert $A$-$A$ bimodule $E$, we endow $E_{\infty} = \ell_\infty(E)/c_0(E)$ with the natural Hilbert $A_{\infty}$-bimodule structure and we identify $E$ with the image of constant sequences in $E_{\infty}$.

For any C$^*$-subalgebra $B\subset A_{\infty}$ the relative commutant and annihilator are given by
\begin{align}
      A_{\infty}\cap B' &:=\{a\in A_{\infty}: ab=ba\ \forall b\in B\},\\
      A_{\infty}\cap B^{\perp}&:=\{a \in A_{\infty}: aB=Ba=0\}.
\end{align}
Following Kirchberg (see \cite{KIRCH06}), we consider the quotient 
\begin{equation}
    F(B,A_{\infty}):=(A_{\infty}\cap B')/A_{\infty}\cap B^{\perp}.
\end{equation}
Of special interest is the case where $B$ is the embedding of  $A$ into $A_{\infty}$ as constant sequences. In this case we denote $F(A,A_{\infty})=F(A)$ and call it the \emph{Kirchberg central sequence algebra} of $A$. Note that $F(A)$ is unital whenever $A$ is $\sigma$-unital. Indeed, any sequential approximate unit $(e_n)_{n\in \N}$  for $A$ represents the identity in $F(A)$.

We now introduce a subalgebra of Kirchberg's central sequence algebra which plays the role of the fixed points algebra of $F(A)$ under the action of a unitary tensor category $(G,J):\C \curvearrowright A$.

Firstly, given a Hilbert $A$-bimodule $E$, We set
 \begin{equation}
     A_{\infty}\cap E':=\{a\in A_\infty\cap A':a\rhd x-x\lhd a=0 \, \text{for all}\ x\in E\subseteq E_{\infty}\}.
 \end{equation}
Note that if $E$ is the trivial $A$-$A$ bimodule this agrees with usual definition of $A_\infty \cap A'$. In general, 
$A_{\infty}\cap E'$ is a norm closed subalgebra of $A_{\infty}$. If we also consider the dual bimodule $\overline{E}$, we get a C$^*$-subalgebra.

\begin{lemma}\label{lem:centralisingalg}
    Let $A$ be a C$^*$-algebra and $E$ be a dualisable Hilbert $A$-bimodule. Then the norm closed algebra
    $(A_{\infty}\cap E')\cap (A_{\infty}\cap \overline{E}')$
    is a C$^*$-algebra.
\end{lemma}
\begin{proof}
    Let $B=(A_{\infty}\cap E')\cap (A_{\infty}\cap \overline{E}')$, as the intersection of two closed subalgebras of $A_\infty$, it follows that $B$ is also a closed subalgebra of $A_{\infty}$. It remains to show that $B$ is a $^*$-subalgebra of $A_{\infty}$. Let $a\in B$ and $\eta\in E$ then
   \begin{equation}
       \|a^*\rhd \eta-\eta\lhd a^*\|=\|\overline{a^*\rhd \eta-\eta\lhd a^*}\|=\|\overline{\eta}\lhd a-a\rhd \overline{\eta}\|=0
   \end{equation}
   as $a\in A_{\infty}\cap \overline{E}'$. Similarly, $a^*\rhd x-x\lhd a^*=0$ for all $x\in \overline{E}$. 
\end{proof}

It's easy to see that, when $E$ is non-degenerate, the annihilator $A_\infty \cap A^\perp$ is always a subalgebra of $A_{\infty}\cap E'$ and that, when $A$ is also $\sigma$-unital, any approximate unit $(e_n)_{n=1}^\infty$ defines an element of $A_{\infty}\cap E'$.
We are now ready to define our subalgebra of Kirchberg's central sequence algebra.
\begin{proposition}\label{prop:fixedcentralseqalg}
    Let $\mathcal{C}$ be a unitary tensor category and $(G,J):\mathcal{C}\rightarrow \Corr_0(A)$ an action. Then 
    \begin{equation}\label{eqn:fixedcentralseqalg}
        F(A)^G:=\frac{\bigcap_{X\in \C} A_{\infty}\cap G(X)'}{A_{\infty}\cap A^{\perp}}
    \end{equation}
    is a C$^*$-algebra. Moreover, if $A$ is $\sigma$-unital $F(A)^G$ is unital.
\end{proposition}
\begin{proof}
 As $\C$ is rigid, any $X\in \C$ has a dual object $(\overline{X},R,\overline{R})$. So each correspondence $G(X)$ also has a dual object in $\Corr_0(A)$ given by $(G(\overline{X}),J_{\overline{X},X}^* G(R), J_{X,\overline{X}}^* G(\overline{R}))$. Therefore, $G(\overline{X})\cong \overline{G(X)}$ and it follows from Lemma \ref{lem:centralisingalg} that $\bigcap_{X\in \C} A_{\infty}\cap G(X)'$ is a C$^*$-subalgebra of $A_\infty \cap A'$. 
 
 Since every object in $\Corr_0(A)$ is a non-degenerate bimodule, it follows that $A_{\infty}\cap A^{\perp}\subseteq A_{\infty}\cap G(X)'$ for all $X\in \C$. Hence, $A_{\infty}\cap A^{\perp}$ is an ideal of $\bigcap_{X\in \C} A_{\infty}\cap G(X)'$. Hence the quotient $F(A)^G$ is well defined as a subalgebra of $F(A)$.
 
 Suppose $A$ is $\sigma$-unital and let $(e_n)_{n=1}^\infty$ be an approximate unit. Then $(e_n)_{n=1}^\infty$ represents the unit of $F(A)$.
 Since every object in $\Corr_0(A)$ is a non-degenerate bimodule, the subalgebra $F(A)^G$ contains the unit.
\end{proof} 
\par If $A$ is unital the annihilator $A_{\infty}\cap A^{\perp}$ is trivial and we simply denote the C$^*$-algebra of Proposition \ref{prop:fixedcentralseqalg} by $(A_{\infty}\cap A')^G$.
\par In general the intersection in \eqref{eqn:fixedcentralseqalg} is taken over the proper class of all $X \in \C$, so must be interpreted with a bit of care. 
Fortunately, there are a number of simplifications: 
As the category $\C$ is semisimple, the intersection can be taken over all simple objects $X \in \C$;
since $A_{\infty}\cap E'$ only depends on the isomorphism class of the Hilbert $A$-bimodule $E$, the intersection can be taken over a isomorphism classes of simple objects. 
In most examples of interest, the category $\C$ will only have countably many isomorphism classes of simple objects. 
If $\C$ is a unitary fusion category. The intersection can be taken to be finite.

We end this section, by showing that for an action of $\Hilb(\Gamma)$, we recover the fixed points of $F(A)$ under the induced action. 
\begin{example}
    Let $\Gamma$ be a discrete group and $\mathcal{C}=\Hilb(\Gamma)$ be acting by automorphisms on a C$^*$-algebra $A$. That is, the $\Hilb(\Gamma)$ action is given by a cocycle action $(\alpha,u)$ and sends $\mathbb{C}_\gamma\mapsto {}_{\alpha_\gamma} A$ for $\gamma\in \Gamma$. Then 
    $a\in \bigcap_{\gamma\in \Gamma}A_{\infty}\cap ({}_{\alpha_\gamma} A)'$ if and only if $(\alpha_\gamma(a)-a)b=0$ for all $b\in A$. Therefore, $F(A)^{(\alpha,u)}$ coincides with the fixed point algebra $F(A)^{\alpha}$ of those central sequences fixed by the automorphisms $\alpha_\gamma$ for all $\gamma\in \Gamma$ modulo $A_{\infty}\cap A^{\perp}$.
\end{example}
\section{A \texorpdfstring{$\C$}{C}-equivariant McDuff type result}
For a Hilbert $A$-bimodule $E$ and a Hilbert $B$-bimodule $H$ we denote their external tensor product by $E\otimes H$ which is a Hilbert $A\otimes B$-bimodule (see \cite{Hilbertmodules}), where $A\otimes B$ denotes the minimal tensor product of the C$^*$-algebras $A$ and $B$.
\begin{definition}\label{def:inducingactions}
Let $(G,J):\C\curvearrowright A$ be an action and $B$ be a C$^*$-algebra. We denote by $G\otimes \id_B:\C\rightarrow\Corr_0(A\otimes B)$ the $\mathbb{C}$-linear functor defined by
\begin{align}
    (G\otimes \id_B)(X)&:=G(X)\otimes B,\\
    (G\otimes \id_B)(f)&:=G(f)\otimes \id_B.
\end{align}
for $X,Y \in \C$ and $f\in \Hom(X,Y)$.
For each $X,Y\in \C$, we denote by 
\begin{equation}
    (J\otimes 1_B)_{X,Y}:(G\otimes \id_B)(Y)\boxtimes (G\otimes \id_B)(X)\rightarrow (G\otimes \id_B)(X\otimes Y) 
\end{equation}
the unitary natural isomorphism defined by
\begin{align}    
    (J\otimes 1_B)_{X,Y}((y\otimes b_1) \boxtimes(x\otimes b_2))&:=J_{X,Y}(y\boxtimes x)\otimes b_1b_2
\end{align}
for $x\in G(X),y\in G(Y)$ and $b_1,b_2\in B$. It is a routine check that this data defines an action $(G\otimes \id_B, J\otimes 1_B):\C \curvearrowright A\otimes B$.
\end{definition}
\begin{remark}\label{rmk:automorphismpic}
If $\alpha:\Gamma\curvearrowright A$ is a group action then it induces an action of $\Hilb(\Gamma)$ which we denote by $(\alpha,1_A)$. It is straightforward to see that $(\alpha\otimes \id_B, 1_A \otimes 1_B)$ is a $\Hilb(\Gamma)$ action on $A \otimes B$ induced by the $\Gamma$ action $\alpha\otimes \id_B:G\curvearrowright A\otimes B$. 
\end{remark}

Following \cite[Definition 1.3]{TOWI07}, a unital C$^*$-algebra $\D  \not\cong \mathbb{C}$ is \emph{strongly self-absorbing} if there is an isomorphism $\D \cong \D \otimes \D$ that is approximately unitary equivalent to the first factor embedding.\ In light of Remark \ref{rmk:automorphismpic} it makes sense to introduce the following generalisation of \cite[Definition 1.7]{WOU23}.
\begin{definition}\label{def:eqstab}
    Let $(G,J):\mathcal{C}\curvearrowright A$ be an action and $\mathcal{D}$ be a strongly self-absorbing C$^*$-algebra. We say $(G,J)$ is \emph{equivariantly $\mathcal{D}$-stable} if $(G,J)$ is cocycle conjugate to $(G\otimes \id_{\D},J\otimes 1_{\D})$, i.e.\ there exists a cocycle conjugacy between $(G,J)$ and $(G\otimes \id_{\D},J\otimes 1_{\D})$. 
\end{definition}
We may now state the main result of this paper.
\begin{theorem}\label{thm:eqMcDuff}
Let $A$ be a separable C$^*$-algebra, $\C$ be a unitary tensor category with countably many isomorphism classes of simple objects and $(G,J):\C\curvearrowright A$ be an action. Let $\D$ be a strongly self-absorbing C$^*$-algebra. The following are equivalent:
\begin{enumerate}[label=\textit{(\roman*)}]
    \item $(A,G,J)$ is equivariantly $\D$-stable i.e. $(A,G,J)$ is cocycle conjugate to $(A\otimes \D,G\otimes \id_{\D},J\otimes 1_{\D})$;
    \label{item:eqMcDuff1}
    \item $(A,G,J)$ is cocycle conjugate to $(A\otimes\D^{\otimes\infty},G\otimes \id_{\D^{\otimes \infty}},J\otimes 1_{\D^{\otimes\infty}})$;\label{item:eqMcDuff2}
    \item there exists a unital embedding  $\D\rightarrow F(A)^G$;\label{item:eqMcDuff3}
    \item the cocycle morphism $(\id_A\otimes 1_{\mathcal{D}},\{\id_{G(X)}\otimes 1_{\mathcal{D}}\}_{X\in \C})$ from $(A,G,J)$ to $(A\otimes\D, G\otimes \id_{\mathcal{D}},J\otimes 1_{\D})$
    is approximately unitarily equivalent to a cocycle conjugacy.\label{item:eqMcDuff4}
\end{enumerate}
\end{theorem}
\begin{proof}
    $\ref{item:eqMcDuff1}\Rightarrow\ref{item:eqMcDuff2}$: Let $(\varphi,\{h^X\}_{X\in \C})$ be a cocycle conjugacy from $(A,G,J)$ to $(A\otimes\D,G\otimes\id_{\D},J\otimes 1_{\D})$. Denote by $\psi:\D\rightarrow \D^{\otimes\infty}$ an isomorphism. The cocycle morphism 
    $(\id_A\otimes \psi,\{\id_{G(X)}\otimes\psi\}_{X\in \C})$ is invertible as $\psi$ is invertible. Therefore the composition 
    $$(\id_A\otimes \psi,\{\id_{G(X)}\otimes\psi\}_{X\in \C})\circ (\varphi,\{h^X\}_{X\in \C})$$ 
    is the desired cocycle conjugacy.
    
    $\ref{item:eqMcDuff2}\Rightarrow\ref{item:eqMcDuff3}$: If $(\varphi,\{h^X\}_{X\in \C})$ is a cocycle conjugacy between two $\C$-actions $G$ and $H$, then $\varphi$ induces an isomorphism from $F(A)^G$ to $F(A)^H$. Therefore, it suffices to show that there is a unital embedding from $\D$ into $F(A\otimes \D^{\otimes\infty})^{G\otimes \id_{\D^{\otimes\infty}}}$.
     \par Let $(h_n)_{n\in \N}$ be a sequential approximate unit for $A$ and denote by $\rho_n:\D\rightarrow \D^{\otimes \infty}$ the unital $^*$-homomorphism taking $\D$ to the $n$-th tensor entry of $\D^{\otimes \infty}$ i.e.\  $\rho_n(d)$ is the image of $1_{\D} \otimes \ldots \otimes 1_{\D} \otimes d\in \D^{\otimes n}$ under the canonical inclusion $\D^{\otimes n}\subset \D^{\otimes \infty}$. 
    Now, the mapping
    \begin{align*}
    \phi:\D&\rightarrow F(A\otimes \D^{\otimes\infty})\\
    d&\mapsto (h_n\otimes \rho_n(d))
    \end{align*}
     is a unital embedding into $F(A\otimes\D^{\otimes\infty})^{G\otimes\id_{\D^{\otimes\infty}}}$, as required.
     \par$\ref{item:eqMcDuff4}\Rightarrow\ref{item:eqMcDuff1}$: Is immediate.
    
    $\ref{item:eqMcDuff3}\Rightarrow\ref{item:eqMcDuff4}$: Let $\psi:\mathcal{D}\rightarrow F(A)^G$ be a unital embedding. We will show that the cocycle morphism $(\id_A\otimes 1_{\mathcal{D}},\{\id_{G(X)}\otimes 1_{\mathcal{D}}\}_{X\in \C})$ satisfies the conditions of \cite[Theorem 6.2]{GINE23} which \ref{item:eqMcDuff4} follows.   
    Define the unital $^*$-homomorphism
    \begin{equation}
        \varphi:\mathcal{D}\otimes \mathcal{D}\rightarrow \frac{\bigcap_{X\in\Irr(\C)}(A\otimes\D)_{\infty}\cap (G(X)\otimes 1_{\mathcal{D}})'}{(A\otimes\D)_{\infty}\cap (A\otimes 1_{\D})^{\perp}}
    \end{equation}
    by $\varphi(d\otimes d')=\psi(d)\otimes d'$.\footnote{We view $\psi(d)\otimes d'$ as an element of the codomain by lifting $\psi(d)$ to an element in $\bigcap_{X\in \Irr(\C)}A\cap G(X)'$, tensoring on $d'$, and then descending to the quotient. It is easy to see that the result is independent of the choice of lift.} 
    For the remainder of the proof we denote the codomain C$^*$-algebra of $\varphi$ by $B$.  That $B$ is indeed a C$^*$-algebra follows in the same way as Proposition \ref{prop:fixedcentralseqalg}, note also that it is a subalgebra of $F(A\otimes 1_{\D},(A\otimes\D)_{\infty})$.\
    \par Firstly, for any $b\in F(A\otimes 1_{\D},(A\otimes\D)_{\infty})$ and $\xi\in G(X)\otimes \D$ with $X\in \Irr(\C)$ we denote by  $b\rhd \xi $ and $\xi\lhd b$ the elements of $(G(X)\otimes \D)_{\infty}$ defined by lifting $b$ to an element of $(A\otimes \D)_\infty \cap (A\otimes 1_{\D})'$ and applying the left and right action by this lift to $\xi$ respectively. That this is independent of the lift follows as $G(X)$ is non-degenerate. 
    Moreover, for $d\in \D$ and $\xi\in G(X)$ one has that
    \begin{equation}\label{1}\varphi(1_{\D}\otimes d)\rhd\xi\otimes 1_{\D}=(\psi(1_{\D})\otimes d)\rhd (\xi\otimes 1_{\D})=\xi\otimes d
    \end{equation}
     as $\psi$ is unital and $G(X)$ is non-degenerate. 
    Similarly, 
    \begin{equation}\label{2}\varphi(d\otimes 1_{\D})\rhd (\xi\otimes 1_{\D})=(\psi(d)\otimes 1_{\D})\rhd (\xi\otimes 1_{\D})\subset G(X)_{\infty}\otimes 1_{\D}.
    \end{equation}
\par Now, let $\mathcal{F}_1\subset \D$ be a finite subset and $\varepsilon>0$. As $\D$ has an approximately inner half flip, there is a unitary $v\in U(\D\otimes\D)$ such that
    \begin{equation}
        \|v^*(1_{\D}\otimes d)v-d\otimes 1_{\D}\|< \varepsilon, \forall d \in \mathcal{F}_1.
    \end{equation}
    Set $u=\varphi(v)$. Let $d\in \mathcal{F}_1$, $X\in \Irr(\C)$ and  $\xi \in G(X)$ with $\|\xi\|\leq 1$. Combining \eqref{1}, \eqref{2} and that $(\xi \otimes 1_{\D})\lhd \varphi(x)=\varphi(x)\rhd (\xi\otimes 1_{\D})$ for all $x\in \D\otimes\D$, we have 
    \begin{equation}
        \begin{split}
     \dist(&G(X)_{\infty}\otimes 1_{\D}, u^*\rhd (\xi\otimes d)\lhd u) \\
      &=\dist(G(X)_{\infty}\otimes 1_{\D},\varphi(v^*)\rhd (\xi \otimes d)\lhd \varphi(v))\\
      &=\dist(G(X)_{\infty}\otimes 1_{\D},\varphi(v^*(1_{\D}\otimes d))\rhd (\xi \otimes 1_{\D})\lhd\varphi(v))\\
      &=\dist(G(X)_{\infty}\otimes 1_{\D},\varphi(v^*(1_{\D}\otimes d)v)\rhd (\xi \otimes 1_{\D}))\\
      &<\dist(G(X)_{\infty}\otimes 1_{\D},\varphi(d\otimes 1_{\D})\rhd (\xi \otimes 1_{\D}))+\varepsilon\\
      &=\varepsilon,
        \end{split}
    \end{equation}
    By \cite[Proposition 1.13]{TOWI07}, we may assume that the unitary $v \in \D \otimes \D$ is homotopic to the identity.\footnote{The $K_1$-injectivity hypothesis of \cite[Proposition 1.13]{TOWI07} is now known to be automatic (see \cite[Remark 3.3]{WIN11}).} Hence,  $u=\varphi(v)$ is also homotopic to the identity. 
    By definition, there is a unital inclusion
    \begin{equation}
        B\subset \frac{(M(A)\otimes\D)_{\infty}}{(A\otimes\D_{\infty})\cap (A\otimes 1_{\D})^{\perp}}.
    \end{equation}
    Since $u$ is homotopic to to the identity, we may lift $u$ to a unitary $w$ in $(M(A)\otimes\D)_{\infty}$. As $w$ lifts $u$, it follows that 
    $$w\rhd (\xi \otimes 1_{\D})=(\xi\otimes 1_{\D})\lhd w$$ and $$\dist(G(X)_{\infty}\otimes 1_{\D},w^*\rhd(\xi \otimes d)\lhd w)<\varepsilon$$ for all $d\in \mathcal{F}_1$ and $\xi\in G(X)$ with $\|\xi\|\leq 1$.\footnote{Here we view $G(X)$ as a Hilbert $M(A)$-module by \cite[Lemma 1.11]{GINE23}. The actions of $w$ on $G(X)\otimes 1_{\D}$ are independent of the lift of $v$ as $G(X)$ is non-degenerate.} 
    Now, pick a sequence of unitaries $(w_n)_{n=0}^\infty$ with $w_n\in M(A)\otimes \D$ that represents $w$. 
    For any finite sets $K\subset \Irr(\C)$ containing $1_{\C}$, $\mathcal{F}_1\subset \D$, $\G_{X}\subset G(X)$ one may choose $n$ large enough such that 
    \begin{align}
    \|w_n& \rhd (\xi \otimes 1_{\D})-(\xi\otimes 1_{\D})\lhd w_n\|\leq \varepsilon,\ \text{for all}\ \xi \in \G_X, X\in K,
    \end{align}
    \begin{align}
     \dist&(G(X)\otimes 1_{\D},w_n^*\rhd (\xi\otimes d)\lhd w_n)\leq \varepsilon,\ \text{for all}\ \xi\in \G_X, \\ 
     \quad& d\in \mathcal{F}_1, \text{and}\ X\in K. \notag
    \end{align}
    Therefore, the cocycle morphism consisting of the $^*$-homomorphism $\id_A\otimes 1_{\D}$ with family of linear maps $h^X=\id_{G(X)}\otimes 1_{\D}$ for $X\in \C$ satisfies the conditions of \cite[Theorem 6.2]{GINE23} and \ref{item:eqMcDuff4} follows.
\end{proof}

\section{Application to \texorpdfstring{$\D$}{D}-stability for stationary AF-actions }\label{sec:example}

In this section we apply our main result to obtain a useful, sufficient condition for large class of AF-actions to be $\D$-stable, where $\D$ is a strongly self-absorbing algebra. First, we introduce some notation and recall the enriched Bratteli diagram formalism from \cite[Definition 4.10]{CHROJO22}.

If $\C$ is a unitary fusion category, then an AF-action is a action on an AF-algebra which is an inductive limit of actions on finite dimensional C$^*$-algebras.\ These can be described by choosing, for each $n\in \mathbbm{N}$, a unitary finitely semisimple (left) $\mathcal{C}$-module category $\mathcal{M}_{n}$, a generating object $m_{n}\in \mathcal{M}_{n}$, and a $\C$-module functor $F_{n}: \mathcal{M}_{n}\rightarrow \mathcal{M}_{n+1}$ such that $F_{n}(m_{n})\cong m_{n+1}$. Note that $m_{n}$ is determined up to isomorphism by the choices of $\mathcal{M}_{n}, F_{n}$ and $m_{0}$. This data is called an \textit{enriched Bratelli diagram} for an AF-action and completely determines an AF-action up to equivalence (though it is not unique).

\begin{definition}
    An AF-action of $\C$ is \textit{stationary} if it has an enriched Bratteli diagram with $\mathcal{M}_{n}=\mathcal{M}_{0}$ and $F_{n}=F_{0}$ for all $n$.
\end{definition}

These are the type of AF-actions that typically arise from reconstruction results in subfactor theory \cite{MR1473221}. 

Given a stationary AF-action of $\C$ with module category $\mathcal{M}$ and module functor $F:\mathcal{M}\rightarrow \mathcal{M}$, we can view $\mathcal{M}$ as a $\mathcal{C}-\C^{*}_{\mathcal{M}}$ bimodule category, where $\C^{*}_{\mathcal{M}}:= \text{End}_{\C}(\mathcal{M})^{op}$ is the (unitary) dual category to $\C$\ \cite[Chapter 7]{MR3242743}. Note that in general $\C^{*}_{\mathcal{M}}$ will be an indecomposable multi-fusion category, which is fusion if and only if $\mathcal{M}$ itself is indecomposable \cite[Section 7.12]{MR3242743}.\ Then we can associate to $F$ an object $Y\in \C^{*}_{\mathcal{M}}$, such that the functor $F$ is equivalent to $\cdot \triangleleft Y  $, and the module functor structure on $F$ is given by the bimodule associator associated to $Y$. For convenience, by MacLane's coherence theorem for 2-categories (e.g. \cite[Theorem 8.4.1]{MR4261588}) we can assume that $\C^{*}_{\mathcal{M}}$ and $\mathcal{C}$ are strict tensor categories and that $\mathcal{M}$ is a strict $\mathcal{C}-\C^{*}_{\mathcal{M}}$ bimodule category.

We now unpack the definition of the inductive limit action from the enriched Bratteli diagram. We set $A_{n}:=\text{End}_{\mathcal{M}}( m_{0}\triangleleft Y^{\otimes n})$ for each $n \in \N$, which are finite dimensional C$^*$-algebras.

Next we define an action of $\C$ on $A_n$ for each $n \in \N$. The functor $G_{n}:\C\rightarrow \Corr(A_{n})$ is defined by 
\begin{equation}
    G_{n}(X):=\text{Hom}_{\mathcal{M}}( m_{0}\triangleleft Y^{\otimes n}, X\triangleright m_{0}\triangleleft Y^{\otimes n}),
\end{equation}
where the $A_n$-bimodule structure is given by $x\cdot f\cdot y:= (1_{X}\triangleright x)\circ f\circ y$ and the right $A_{n}$-valued inner product is given by $\langle f , g\rangle_{A_n}:=f^{*}\circ g$. The tensorator $J^{X_1,X_2}_n:G_n(X_2) \boxtimes G_n(X_1) \rightarrow G_n(X_1 \otimes X_2)$ is defined on elementary tensors via the composition 
\begin{equation}
     m_{0} \triangleleft Y^{\otimes n} \xrightarrow{f_1}  
      X_1 \triangleright m_{0} \triangleleft  Y^{\otimes n}
      \xrightarrow{1_{X_1} \triangleright f_2 } X_1 \triangleright X_2 \triangleright m_{0} \triangleleft  Y^{\otimes n}, 
\end{equation}
for $f_1 \in G_n(X_1)$ and $f_2 \in G_n(X_2)$, as by strictness $X_1 \triangleright X_2 \triangleright m_{0} \triangleleft  Y^{\otimes n} = X_1 \otimes X_2 \triangleright m_{0} \triangleleft  Y^{\otimes n}$.

We define connecting maps $\varphi_n:A_n \rightarrow A_{n+1}$ via $x\mapsto x\triangleleft 1_{Y}$, and connecting cocycle morphisms $(\varphi_n,\{j_n^X\}_{X\in \C}):(G_n,A_n)\rightarrow (G_{n+1},A_{n+1})$ where $j^{X}_{n}(f):=f\triangleleft 1_{Y}$. 
The result is an AF action $(G,J)$ of $\C$ on the unital AF algebra  $A = \displaystyle\lim_{\to} (A_{n},\varphi_n)$.

Now, before we state our main result of this section, for any (strict) unitary tensor category $\mathcal{E}$ and any object $Y\in \mathcal{E}$, we define the AF-algebra $B^{Y}$ as the inductive limit of $B^{Y}_{n}:=\text{End}_{\mathcal{E}}(Y^{\otimes n})$ with connecting maps $x\mapsto x \otimes 1_{Y}$.

\begin{theorem}\label{thm:BYintrocentralsequence}
Let $\C$ be a unitary fusion category and $\mathcal{M}$ a unitary indecomposable $\mathcal{C}$-module category with generating object $m_0\in \mathcal{M}$. Let $Y \in \C^{*}_{\mathcal{M}}$, and consider the stationary AF action $G:\C\rightarrow \Corr(A)$ defined by $(m_0,\mathcal{M},Y)$. Then $(A_{\infty}\cap A^{\prime})^{G}$ contains $B^{Y}=\lim \operatorname{End}_{\C^{*}_{\mathcal{M}}}(Y^{\otimes n})$ as a unital subalgebra. In particular, if $B^{Y}$ contains a strongly self-absorbing unital subalgebra $\mathcal{D}$, then the action $G$ is $\D$-stable.
\end{theorem}

\begin{proof}
Let $x\in B^{Y}_{m}=\text{End}_{\C^{*}_{\mathcal{M}}}(Y^{\otimes m})$. Define $x_{n}:= 1_{m_{0} \triangleleft Y^{\otimes n}} \triangleleft x \in \text{End}(m_0  \triangleleft Y^{\otimes (n+m)}):=A_{n+m}$. 
By construction, $x_n$ commutes with the image of $G_{n}(X)$ under $j_n^{X}\circ\cdots\circ j^X_{n+m-1}$. 
Hence, the the map $x \mapsto (x_n)_{n=1}^\infty$ induces a unital injective $^*$-homomorphism $B^{Y}_{m}\rightarrow (A_{\infty}\cap A^{\prime})^{G}$ for each $m \in \N$. As these maps are compatible with the connecting maps $x\mapsto x \otimes 1_{Y}$, we obtain a unital injective $^*$-homomorphism $B^{Y}\rightarrow (A_{\infty}\cap A^{\prime})^{G}$.
\end{proof}

We call a stationary AF-action satisfying the hypotheses of the above theorem an \textit{indecomposable stationary AF-action}. Recall an object $Y$ in a fusion category $\mathcal{E}$ is called a \textit{strong tensor generator} of $\mathcal{E}$ if there exists an $n \in \N$ such that every  simple object of $\mathcal{E}$ is isomorphic to a summand of $Y^{\otimes n}$.  

\begin{corollary}\label{cor:stationaryareZstable} Let $G$ be the indecomposable stationary AF-action of the unitary fusion category $\mathcal{C}$ on the unital AF-algebra $A$ determined by the triple $(m_0,\mathcal{M},Y)$. If $Y$ is a strong tensor generator for $\C^{*}_{\mathcal{M}}$, then $G$ is $\mathcal{Z}$-stable.
\end{corollary}
\begin{proof}
    Since $Y$ is a strong tensor generator, 
    there exists an $n \in \N$ such that every  simple object of $\C^{*}_{\mathcal{M}}$ is isomorphic to a summand of $Y^{\otimes n}$.
    Therefore, the Bratteli diagram for $B^{Y}$ given by tensoring $Y^{\otimes n}$ at each stage has an adjacency matrix with all entries non-zero. It follows that $B^{Y}$ is a simple AF-algebra (see for example \cite[Chapter 6]{MR0623762}). 
    Hence, $B^{Y}$ is $\Z$-stable.
    The result now follows by combining Theorem \ref{thm:BYintrocentralsequence} with Theorem \ref{thm:eqMcDuff}.
\end{proof}

\bibliographystyle{abbrv}

\begin{thebibliography}{10}

\bibitem{YUKAKU23}
Y.~Arano, K.~Kitamura, and Y.~Kubota.
\newblock Tensor category equivariant {KK}-theory.
\newblock {\em Adv. Math.}, 453:109848, 2024.

\bibitem{Class1}
J.~R. Carri\'on, J.~Gabe, C.~Schafhauser, A.~Tikuisis, and S.~White.
\newblock Classifying $^*$-homomorphisms {I}: {U}nital simple nuclear
  {C}$^*$-algebras.
\newblock {\em arXiv preprint arXiv:2307.06480}, 2023.

\bibitem{DIMNUC21}
J.~Castillejos, S.~Evington, A.~Tikuisis, S.~White, and W.~Winter.
\newblock Nuclear dimension of simple {$\rm C^*$}-algebras.
\newblock {\em Invent. Math.}, 224(1):245--290, 2021.

\bibitem{MR4419534}
Q.~Chen, R.~Hern\'{a}ndez~Palomares, C.~Jones, and D.~Penneys.
\newblock Q-system completion for {$\rm C^*$} 2-categories.
\newblock {\em J. Funct. Anal.}, 283(3):Paper No. 109524, 59, 2022.

\bibitem{CHROJO22}
Q.~Chen, R.~H. Palomares, and C.~Jones.
\newblock {K}-theoretic classification of inductive limit actions of fusion
  categories on {AF}-algebras.
\newblock {\em Comm. Math. Phys.}, 405(3):Paper No. 83, 52, 2024.

\bibitem{MR0623762}
E.~G. Effros.
\newblock {\em Dimensions and {$C\sp{\ast} $}-algebras}, volume~46 of {\em CBMS
  Regional Conference Series in Mathematics}.
\newblock Conference Board of the Mathematical Sciences, Washington, DC, 1981.

\bibitem{EGLN15}
G.~A. Elliott, G.~Gong, H.~Lin, and Z.~Niu.
\newblock On the classification of simple amenable {C$^*$}-algebras with finite
  decomposition rank, {II}.
\newblock {\em J. Noncommut. Geom.}, 2024.

\bibitem{MR3242743}
P.~Etingof, S.~Gelaki, D.~Nikshych, and V.~Ostrik.
\newblock {\em Tensor categories}, volume 205 of {\em Mathematical Surveys and
  Monographs}.
\newblock American Mathematical Society, Providence, RI, 2015.

\bibitem{MR1642584}
D.~E. Evans and Y.~Kawahigashi.
\newblock {\em Quantum symmetries on operator algebras}.
\newblock Oxford Mathematical Monographs. The Clarendon Press, Oxford
  University Press, New York, 1998.
\newblock Oxford Science Publications.

\bibitem{EVGI21}
S.~Evington and S.~Gir\'{o}n~Pacheco.
\newblock Anomalous symmetries of classifiable {C}$^*$-algebras.
\newblock {\em Studia Math.}, 270(1):73--101, 2023.

\bibitem{GASZ22}
J.~Gabe and G.~Szab{\'o}.
\newblock The dynamical {K}irchberg--{P}hillips theorem.
\newblock {\em Acta Math. 232 (1)}, 1--77, 2024.

\bibitem{GAHIVA22}
E.~Gardella, I.~Hirshberg, and A.~Vaccaro.
\newblock Strongly outer actions of amenable groups on {$\mathcal{Z}$}-stable
  nuclear {$C^*$}-algebras.
\newblock {\em J. Math. Pures Appl. (9)}, 162:76--123, 2022.

\bibitem{GINE23}
S.~Gir\'on~Pacheco and R.~Neagu.
\newblock An {E}lliott intertwining approach to classifying action of
  {C}$^*$-tensor categories.
\newblock {\em arXiv preprint arXiv:2310.18125}, 2023.

\bibitem{GLN15}
G.~Gong, H.~Lin, and Z.~Niu.
\newblock Classification of finite simple amenable $\mathcal{Z}$-stable
  {C$^*$}-algebras {I},{II}.
\newblock {\em C. R. Math. Acad. Sci. Soc. R. Can.}, 42(3-4):63--539, 2020.

\bibitem{IZU02}
M.~Izumi.
\newblock Inclusions of simple {$C^\ast$}-algebras.
\newblock {\em J. Reine Angew. Math.}, 547:97--138, 2002.

\bibitem{IZU04}
M.~Izumi.
\newblock Finite group actions on {$C^*$}-algebras with the {R}ohlin property.
  {I}.
\newblock {\em Duke Math. J.}, 122(2):233--280, 2004.

\bibitem{IZMA21}
M.~Izumi and H.~Matui.
\newblock Poly-{$\mathbb{Z}$} group actions on {K}irchberg algebras {I}.
\newblock {\em Int. Math. Res. Not. IMRN}, 2021(16):12077--12154, 2021.

\bibitem{Jiang-Su}
X.~Jiang and H.~Su.
\newblock On a simple unital projectionless {$C^*$}-algebra.
\newblock {\em Amer. J. Math.}, 121(2):359--413, 1999.

\bibitem{MR4261588}
N.~Johnson and D.~Yau.
\newblock {\em 2-dimensional categories}.
\newblock Oxford University Press, Oxford, 2021.

\bibitem{MR1473221}
V.~Jones and V.~S. Sunder.
\newblock {\em Introduction to subfactors}, volume 234 of {\em London
  Mathematical Society Lecture Note Series}.
\newblock Cambridge University Press, Cambridge, 1997.

\bibitem{MR3166042}
V.~F.~R. Jones, S.~Morrison, and N.~Snyder.
\newblock The classification of subfactors of index at most 5.
\newblock {\em Bull. Amer. Math. Soc. (N.S.)}, 51(2):277--327, 2014.

\bibitem{KWP04}
T.~Kajiwara, C.~Pinzari, and Y.~Watatani.
\newblock Jones index theory for {H}ilbert {C$^*$}-bimodules and its
  equivalence with conjugation theory.
\newblock {\em J. Funct. Anal.}, 215(1):1--49, 2004.

\bibitem{KW00}
T.~Kajiwara and Y.~Watatani.
\newblock Jones index theory by {H}ilbert {C$^*$}-bimodules and {$K$}-theory.
\newblock {\em Trans. Amer. Math. Soc.}, 352(8):3429--3472, 2000.

\bibitem{KIRCH95}
E.~Kirchberg.
\newblock Exact {${\rm C}^*$}-algebras, tensor products, and the classification
  of purely infinite algebras.
\newblock In {\em Proceedings of the {I}nternational {C}ongress of
  {M}athematicians, {V}ol. 1, 2 ({Z}\"{u}rich, 1994)}, pages 943--954.
  Birkh\"{a}user, Basel, 1995.

\bibitem{KIRCH06}
E.~Kirchberg.
\newblock Central sequences in {$C^*$}-algebras and strongly purely infinite
  algebras.
\newblock In {\em Operator {A}lgebras: {T}he {A}bel {S}ymposium 2004}, volume~1
  of {\em Abel Symp.}, pages 175--231. Springer, Berlin, 2006.

\bibitem{KIPI00}
E.~Kirchberg and N.~C. Phillips.
\newblock Embedding of exact {$C^*$}-algebras in the {C}untz algebra
  {$\mathcal{O}_2$}.
\newblock {\em J. Reine Angew. Math.}, 525:17--53, 2000.

\bibitem{Hilbertmodules}
E.~C. Lance.
\newblock {\em Hilbert {$C^*$}-modules}, volume 210 of {\em London Mathematical
  Society Lecture Note Series}.
\newblock Cambridge University Press, Cambridge, 1995.
\newblock A toolkit for operator algebraists.

\bibitem{LONRO97}
R.~Longo and J.~E. Roberts.
\newblock A theory of dimension.
\newblock {\em $K$-Theory}, 11(2):103--159, 1997.

\bibitem{MASA12}
H.~Matui and Y.~Sato.
\newblock {$\mathcal{Z}$}-stability of crossed products by strongly outer
  actions.
\newblock {\em Comm. Math. Phys.}, 314(1):193--228, 2012.

\bibitem{MASA14}
H.~Matui and Y.~Sato.
\newblock {$\mathcal{Z}$}-stability of crossed products by strongly outer
  actions {II}.
\newblock {\em Amer. J. Math.}, 136(6):1441--1496, 2014.

\bibitem{MR1966524}
M.~M\"{u}ger.
\newblock From subfactors to categories and topology. {I}. {F}robenius algebras
  in and {M}orita equivalence of tensor categories.
\newblock {\em J. Pure Appl. Algebra}, 180(1-2):81--157, 2003.

\bibitem{Phil00}
N.~C. Phillips.
\newblock A classification theorem for nuclear purely infinite simple
  {C$^*$}-algebras.
\newblock {\em Doc. Math.}, 5:49--114, 2000.

\bibitem{MR1278111}
S.~Popa.
\newblock Classification of amenable subfactors of type {II}.
\newblock {\em Acta Math.}, 172(2):163--255, 1994.

\bibitem{MR1334479}
S.~Popa.
\newblock An axiomatization of the lattice of higher relative commutants of a
  subfactor.
\newblock {\em Invent. Math.}, 120(3):427--445, 1995.

\bibitem{SZ18Kirchberg}
G.~Szab\'{o}.
\newblock Equivariant {K}irchberg-{P}hillips-type absorption for amenable group
  actions.
\newblock {\em Comm. Math. Phys.}, 361(3):1115--1154, 2018.

\bibitem{SZI}
G.~Szab\'{o}.
\newblock Strongly self-absorbing {$\rm C^*$}-dynamical systems.
\newblock {\em Trans. Amer. Math. Soc.}, 370(1):99--130, 2018.

\bibitem{TWW17}
A.~Tikuisis, S.~White, and W.~Winter.
\newblock Quasidiagonality of nuclear {C$^*$}-algebras.
\newblock {\em Ann. of Math. (2)}, 185(1):229--284, 2017.

\bibitem{TO08}
A.~S. Toms.
\newblock On the classification problem for nuclear {$C^\ast$}-algebras.
\newblock {\em Ann. of Math. (2)}, 167(3):1029--1044, 2008.

\bibitem{TOWI07}
A.~S. Toms and W.~Winter.
\newblock Strongly self-absorbing {$C^*$}-algebras.
\newblock {\em Trans. Amer. Math. Soc.}, 359(8):3999--4029, 2007.

\bibitem{WA90}
Y.~Watatani.
\newblock Index for {$C^*$}-subalgebras.
\newblock {\em Mem. Amer. Math. Soc.}, 83(424):vi+117, 1990.

\bibitem{WIN11}
W.~Winter.
\newblock Strongly self-absorbing {$C^*$}-algebras are {$\mathcal{Z}$}-stable.
\newblock {\em J. Noncommut. Geom.}, 5(2):253--264, 2011.

\bibitem{WOU23}
L.~Wouters.
\newblock Equivariant {$\mathcal{Z}$}-stability for single automorphisms on
  simple {$C^*$}-algebras with tractable trace simplices.
\newblock {\em Math. Z.}, 304(1):Paper No. 22, 36, 2023.

\end{thebibliography}

\end{document}